\documentclass[11pt,svgnames]{amsart}
\usepackage{amsfonts, amssymb, amsmath, amsthm}
\usepackage{mathrsfs} 
\usepackage{latexsym}
\usepackage{enumerate}
\usepackage{multicol}
\usepackage{verbatim}

\usepackage{url}
\usepackage{csquotes}
\usepackage{placeins}
\usepackage{epic}
\usepackage{graphicx}
\usepackage{epstopdf}
\usepackage{pstricks}
\usepackage{pst-plot}
  
\usepackage[colorlinks=true]{hyperref}
\hypersetup{citecolor=blue, linkcolor=blue, urlcolor= blue}
\usepackage{mathabx} 

\input amssym.def
\usepackage{amscd}
\usepackage[mathscr]{eucal}
\usepackage{palatino}
\usepackage{mathtools}

\DeclarePairedDelimiter{\floor}{\lfloor}{\rfloor}
\newfont{\cyrr}{wncyr10}
\usepackage{tikz-cd}
\usepackage[utf8]{inputenc}
\usepackage{amsmath}

\newcommand{\thmref}[1]{Theorem~\ref{#1}}

\newcommand{\lemref}[1]{Lemma~\ref{#1}}

\newcommand{\rO}{{{\rm{O} }}}

\newcommand{\Z}{{\mathbb Z}}
\newcommand{\Q}{{\mathbb Q}}

\newcommand{\R}{{\mathbb R}}

\newcommand{\m}{{\mathfrak m}}

\def\li{\operatorname{Li}}

\newtheorem{thm}{Theorem}

\newtheorem{lem}[thm]{Lemma}
\newtheorem{cor}[thm]{Corollary}

\newtheorem{defn}{Definition}

\title[Diophantine Approximation with primes in an arithmetic progression]
{Diophantine Approximation with primes in an arithmetic progression}
\author{D. Mazumder}
 \address[D. Mazumder]{GITAM (Deemed to be university), Gandhi Nagar, Rushikonda,Vishakapatnam, Andhra Pradesh - 530045, India.}
\email{dwmaz1993@gmail.com}
\author{J. Sivaraman}
  \address[J. Sivaraman]{Indian Institute of Science Education and Research Thiruvananthapuram, Vithura, Kerala - 595551, India.}
\email{jyothsnaas@gmail.com}
\begin{document}

\begin{abstract}
Let $\alpha \in \mathbb{R} \setminus \mathbb{Q}$, $\beta \in \R$, $N \in \mathbb{R}_{\ge 1}$ and $ \Delta \in (0, 1/2)$.
 For any real $y$, let $\|y\|$ denote the distance from $y$ to the nearest integer.
 Further, we use $\Lambda$ to denote the von Mangoldt function.
It is conjectured that for every $\epsilon \in \mathbb{R}_{>0}$, there are infinitely many primes $\ell$ such that
$
\| \alpha \ell \| < \ell^{-1 +\epsilon}.
$
In the first part of this paper, we show that
given two coprime integers $u, v \ge 1$,
there are infinitely many primes  $\ell \equiv u \bmod v$ such that 
$
 \|\alpha \ell - \beta\| \ll_v \ell^{-1/4} \log^{8} \ell.
$
In order to prove this result we first prove the following general theorem and then deduce the above as a corollary.
Before stating the result, let us define  a function on  $f_{\Delta}(\theta)$ on $\R$ such that $f_{\Delta}(\theta)$ is
$ 1 \text{ if } \| \theta \| < \Delta$ and
$ 0 $ otherwise. Further, suppose that $u, v \in \mathbb{Z}_{\ge 1}$ are coprime and $a$, $q \in \Z$ are coprime with $q> N^{1/4}$ and $|\alpha v -a/q| \le 1/q^2$.
Then, for every $\epsilon \in \mathbb{R}_{>0}$ we have
\begin{equation*}
\sum_{\substack{n=1 \\ n \equiv u \bmod v}}^N \Lambda(n) (f_{\Delta}(\alpha n - \beta) - 2\Delta) \ll_v (Nq^{-1/2} + N^{3/4} + N^{5/6}\Delta^{1/2} + (\Delta Nq)^{1/2} +  N^{\epsilon}q \Delta^{1-\epsilon}) \mathcal{L}^8 
\end{equation*}
where $\mathcal{L}=\log (Nq/\Delta)$.
Setting  $N=q^2$ and $\Delta = CN^{-1/4} \log^{8} N$ for a sufficiently large
constant $C>0$, we deduce the result claimed above.
This generalises a well known result of Vaughan from 1977 proving a similar bound for the sum
\begin{equation*}
\sum_{\substack{n=1}}^N \Lambda(n) (f_{\Delta}(\alpha n - \beta) - 2\Delta).
\end{equation*}
In the second part of this paper, we explicitly construct  an uncountable set $S \subset \R\setminus \Q$ such that for every $\gamma \in S$
there are infinitely many primes $\ell \equiv u \bmod v$ satisfying
$\| \gamma \ell \| < \ell^{-1}$. 
Further we prove unconditionally that not
all the elements of $S$ are Liouville numbers.
This addresses a question of Erd{\"o}s and Mahler from 1939.
\end{abstract}
\maketitle
\section{Introduction and statements of the results}

Let $\alpha \in \mathbb{R} \setminus \mathbb{Q}$, $\beta \in \R$, $\delta \in \mathbb{R}_{>0}$, $N \in \mathbb{R}_{\ge 1}$ and $ \Delta \in (0, 1/2)$.
For any real $y$, let $\|y\|$ denote the distance from $y$ to the nearest integer. 
Throughout this article, we set $e(x):=e^{2\pi i x}$. We shall use $\Lambda(n)$ to denote the von Mangoldt function evaluated at $n \in \mathbb{N}$.
In 1954, Vinogradov \cite{GH} proved that for every $\epsilon \in \mathbb{R}_{>0}$ there are infinitely many primes $\ell$ such that $\|\alpha \ell\| < \ell^{-1/5 +\epsilon}$.
This result was improved by Vaughan \cite{RCV}  in 1977. Vaughan's result states the following.
Let $a,q \in \mathbb{Z}$ be coprime with $q > 0$ and $|\alpha -a/q| \le 1/q^2$.
Vaughan defined a $1$ periodic  function $g_{\Delta}(\theta)$ on $\R$ such that for $\theta \in [-1/2,1/2)$, $g_{\Delta}(\theta)$ is 
$ 1 \text{ if } \theta  \in  [-\Delta, \Delta)$ and
$ 0$ otherwise.
Then, for every $\epsilon \in \mathbb{R}_{>0}$, we have
\begin{equation*}
\sum_{\substack{n=1}}^N \Lambda(n) (g_{\Delta}(\alpha n - \beta) - 2\Delta) \ll    (Nq^{-1/2} + N^{3/4} + (\Delta Nq)^{1/2})\log^8 (Nq/\Delta) +   \Delta^{2/5} N^{4/5} (Nq/\Delta)^{\epsilon}.
\end{equation*}
Setting  $N=q^2$ and $\Delta = CN^{-1/4} \log^{8} N$ for a sufficiently large
constant $C>0$, we get that 
there are infinitely many primes $\ell$ such
$
\| \alpha \ell - \beta \| \ll \ell^{-1/4} \log^8 \ell.
$
It is conjectured that  for every $\epsilon \in \mathbb{R}_{>0}$, there are infinitely many primes $\ell$ such that $\|\alpha \ell\| < \ell^{-1 +\epsilon}$. 
In 1996, Harman \cite{GHII} proved that 
 there are infinitely many primes $\ell$ such
$
\| \alpha \ell \| < \ell^{- 7/22}$.
In 2002,
 Heath-Brown and  Jia \cite{HBJ} improved upon Harman's result.
They showed that for every $\epsilon \in \mathbb{R}_{>0}$, there are infinitely many primes $\ell$ such that $\|\alpha \ell\| < \ell^{-16/49 + \epsilon}$. 
The best known result in this direction is due to Matom{\"a}ki \cite{KM}. Matom{\"a}ki proved that for every $\epsilon \in \mathbb{R}_{>0}$,
 there are infinitely many primes $\ell$ such
$
\| \alpha \ell \| < \ell^{-1/3 + \epsilon}.
$
However the techniques of Harman, Heath-Brown and Jia and Matom{\"a}ki only produce
a lower bound for the number of primes $\ell \le N$
for which 
$
\| \alpha \ell \| < \ell^{-\delta}
$
for the corresponding values of $\delta$ and $N \gg 1$
while the result of Vaughan produces an asymptotic.
Throughout the article, ${\rm O}_v$ or $\ll_v$ will be used to denote that
the implied constant depends on $v$.
The main result of this paper is the following.
\begin{thm}\label{asymp}
Let $\alpha \in \R\setminus \Q$, $\beta \in \R$ and $u$, $v \in \mathbb{N}$ with $(u,v)=1$ be fixed. Further, let $N\in \R_{\ge 1}$ and $\Delta \in (0, 1/2)$.
Suppose that $a$, $q \in \Z$ are coprime with $q> N^{1/4}$ and $|\alpha v -a/q| \le 1/q^2$.
Set $\mathcal{L} = \log (Nq/\Delta)$. We define  a  function 
$f_{\Delta}(\theta)$ on $\R$ such that $f_{\Delta}(\theta)$ is 
$ 1 \text{ if } \| \theta \| < \Delta$ and
$ 0$ otherwise.
Then for any $\epsilon \in \R_{>0}$, we have
$$
\sum_{\substack{n=1 \\ n \equiv u \bmod v}}^N \Lambda(n) (f_{\Delta}(\alpha n - \beta) - 2\Delta) \ll_v  \mathcal{L}^8  (N/q^{1/2} + N^{3/4} + N^{5/6}\Delta^{1/2} + (\Delta Nq)^{1/2} +  N^{\epsilon}q \Delta^{1-\epsilon}).
$$
Here $\Lambda(n)$ is used to denote the von Mangoldt function evaluated at $n \in \mathbb{N}$.\end{thm}

As a corollary, we obtain the following result.
\begin{cor}\label{mainthm}
Suppose $\alpha \in \R\setminus \Q$, $\beta \in \R$ and $u$, $v \in \mathbb{N}$ with $(u,v)=1$ are fixed.
There exist infinitely many primes  $\ell \equiv u \bmod v$ such that 
\begin{equation}
 \|\alpha \ell - \beta\| \ll_v \ell^{-1/4} \log^{8} \ell.
\end{equation}
Suppose that $a$, $q \in \Z$ are coprime with $q> N^{1/4}$ and $|\alpha v -a/q| \le 1/q^2$.
Set $\mathcal{L} = \log (Nq/\Delta)$.
Then, for any $\Delta$ in $(0,1/2)$ and any $\epsilon>0$, 
we have
\begin{multline*}
|\{ 1 \le \ell \le N~ : ~\ell \text{ prime},~\ell \equiv u \bmod v,~ \| \alpha \ell - \beta\| < \Delta\}|   = 
{2\Delta \pi(N,u,v)} \\
 + 
{\rm O}_v(\mathcal{L}^7  ( N/q^{1/2} + N^{3/4} + N^{5/6}\Delta^{1/2} + (\Delta Nq)^{1/2} + N^{\epsilon}q \Delta^{1-\epsilon})).
\end{multline*}
Here $\pi(N,u,v)$ is used to denote the number of primes $\ell \le N$ such that $\ell \equiv u 
\bmod v$.
\end{cor}
This is the immediate analogue of Vaughan's \cite{RCV} result in this setup. 
An important result required to prove \thmref{asymp}
is the following bound on exponential sums.
\begin{thm}\label{expsumbound}
Let $\alpha \in \R\setminus \Q$ and $u$, $v \in \mathbb{N}$ with $(u,v)=1$ be fixed. Further, let  $H$ and $N \in \R_{\ge 1}$.
Suppose that $a$, $q \in \Z$ are coprime with $q>N^{1/4}$ and $|\alpha v -a/q| \le 1/q^2$.
Set $\mathcal{L}_1 = \log 2HN$.
Then for any $\epsilon \in \R_{>0}$, we have
$$
D_1:=  \sum_{h=1}^H  \bigg|\sum_{\substack{n=1 \\ n \equiv u \bmod v}}^N \Lambda(n) e( \alpha hn) \bigg| \ll_v  \mathcal{L}_1^{7}(HN/q^{1/2} + HN^{3/4} + H^{1/2}N^{5/6} + (HNq)^{1/2} + (HN)^{\epsilon}q).
 $$
Again, $\Lambda(n)$ is used to denote the von Mangoldt function evaluated at $n \in \mathbb{N}$.\end{thm}
In 1939, Erd{\"o}s and Mahler \cite{EM} studied the properties of convergents of continued fractions
of numbers. In this work, they state that it is easy to find Liouville numbers for which every convergent 
has prime denominator. However they require the Riemann hypothesis to show the existence of non-Liouville
numbers with this property.
In the last section of our paper, we explicitly construct an uncountable set $S \subset \R\setminus \Q$ such that for every $\gamma \in S$
there are infinitely many primes $\ell \equiv u\bmod v$ such that 
$\| \gamma \ell \| < \ell^{-1}$. We conclude by proving unconditionally that not
all the elements of $S$ are Liouville numbers.

The article is divided into four sections. In Section \ref{prereq}, we recall some well known
results which will be used in the proof of our theorems.
In Section \ref{proofofmainthm}, we prove Theorems \ref{expsumbound},  \ref{asymp} and Corollary \ref{mainthm}.
Finally in Section \ref{examples}, we produce 
uncountably many explicit examples of irrational numbers $\gamma$
for which  $
\|\gamma \ell\|< \ell^{-1}
$
is satisfied by infinitely many primes $\ell \equiv u \bmod v$.
We also show that not all of these numbers are Liouville.

\section{Prerequisites}\label{prereq}
In this section, we state the results used in the proof of our theorems.
We begin 
with the following type I and type II estimates
of Vaughan from \cite{RCV}. The type I estimate is as follows.
\begin{lem} \label{Vaughan1}  {\rm[Vaughan, \cite{RCV}]}
Let  $X,Y \ge 1$ and $\alpha_1$ be real numbers. Further suppose that
 $a,q$ are integers
satisfying $(a,q)=1$, $q>0$ and $|\alpha_1 - a/q| \le 1/q^2$. Set $\mathcal{L}_2 = \log 2XYq$. Now consider
$$
T = \sum_{x \le X} \max_{Z \le XY/x} \bigg| \sum_{y \le Z} a_x e(\alpha_1 xy) \bigg|
$$
where the $a_x$ are complex numbers.
Then 
$$
T \ll \mathcal{L}_2 (XYq^{-1} +X + q) \max_{x \le X} |a_x|. 
$$
\end{lem}
We now move on to the type II estimate of Vaughan.
\begin{lem}\label{Vaughan2}{\rm [Vaughan, \cite{RCV}]}
Let  $X$, $Y \ge 1$ and $\alpha_1$ be real numbers. Further suppose that $a$, $q$ are integers
satisfying $(a,q)=1$, $q>0$ and $|\alpha_1 - a/q| \le 1/q^2$.
Set $\mathcal{L}_2 = \log 2XYq$. Now consider
$$
S = \sum_{x \le X} \max_{Z \le Y} \bigg| \sum_{y \le Z} a_xb_y e(\alpha_1 xy) \bigg|,
$$
where the $a_x$ and $b_y$ are complex numbers. Then
$$
S \ll \mathcal{L}_2^{\frac32} \left( \sum_{x \le X} |a_x|^2 \sum_{y \le Y} |b_y|^2 \right)^{\frac12} (XYq^{-1} +X+Y+q)^{\frac12}.
$$
\end{lem}
\noindent
Two other tools that will be used in proving the main theorems are the following.
The first is a version of Perron's formula.
\begin{lem}\label{perron}{\rm [Page 31 of \cite{GH}]}
For $\rho, \gamma > 0$ and $L\ge1$ real numbers with $\gamma \neq \rho$, we have
\begin{eqnarray*}
\frac{1}{\pi} \int_{-L}^L e^{i\gamma t} \frac{ \sin \rho t}{t} dt = \delta  + \rO(L^{-1} |\rho-\gamma|^{-1} ),
~~~\text{ where }~~~
\delta = \begin{cases}
 0 ~\text{ if } \gamma > \rho,\\ 
 1 ~\text{ if } \gamma < \rho. 
\end{cases}
\end{eqnarray*}
\end{lem}
The second tool we will use is the following expansion of a certain indicator function.
\begin{lem}\label{Fexp}{\rm [see \cite{GH}]}
Let $\Delta \in (0, 1/2)$, $L \in \mathbb{N}$ and 
$f_{\Delta}(\theta)$ be a $1$-periodic function  on $\R$ such that $f_{\Delta}(\theta)$ is 
$ 1 \text{ if } \| \theta \| < \Delta$ and
$ 0$ otherwise.
Then there are coefficients $c_{\ell}^{+}$
such that
$$
 f_{\Delta}(\theta) - 2\Delta \le  \sum_{0 < |\ell| \le L} c_{\ell}^{+}e(\ell \theta) + \frac{1}{L+1}
\text{ with }
|c_{\ell}^{+}| \ll \min \left(\Delta + \frac{1}{L}, \frac{1}{|\ell|} \right). 
$$
\end{lem}

With these results, we conclude this section and proceed to the proofs of our theorems.
\section{Unconditional results on approximation by primes in arithmetic progression}\label{proofofmainthm}
Throughout this section $\alpha \in \R\setminus \Q$ is fixed and $u,v$ are fixed positive integers
satisfying $u < v$ and $(u,v)=1$. Further, for any real number $x$, we use $\floor{x}$ to denote the
greatest integer less than or equal to $x$.
We begin by proving \thmref{expsumbound}.

\begin{proof}[Proof of \thmref{expsumbound}]
We first suppose $H > q$, by \lemref{Vaughan2} we get
\begin{eqnarray*}
D_1 =  \sum_{h=1}^H  \bigg|\sum_{\substack{n=1 \\ n \equiv u \bmod v}}^N \Lambda(n) e(\alpha hn) \bigg| 
& = &  \sum_{h=1}^H  \bigg|\sum_{k=0}^{(N-u)/v} \Lambda(u+kv) e(\alpha hkv) \bigg| \\
& \ll & (\log 2HNq)^{3/2} (HN\log^2 N)^{1/2} (HN/q + H + N + q)^{1/2} \\
& \ll & (\log 2HNq)^{5/2} (HN/q^{1/2} + HN^{1/2} + NH^{1/2} + (HNq)^{1/2})\\
& \ll & (\log 2HN)^{5/2} (HN/q^{1/2} + HN^{1/2} + NH^{1/2} + (HNq)^{1/2}),
\end{eqnarray*}
where the last inequality follows from the fact that $H>q$.
Further 
$$
NH^{1/2} \le H^{1/2}N \cdot H^{1/2}/q^{1/2} \ll HN/q^{1/2}.
$$
Hence the result follows if $H>q$. 
If $q>N$ and $q\ge H$  then again 
we apply \lemref{Vaughan2} to get
\begin{eqnarray*}
D_1& \ll & (\log 2HNq)^{5/2} (HN)^{1/2} ((HN/q)^{1/2} + H^{1/2} + N^{1/2} + q^{1/2}) \\
& \ll & (\log 2HNq)^{5/2} (HN/q^{1/2} + HN^{1/2} + (HNq)^{1/2}).
\end{eqnarray*}
Hence the result follows in this case as well. So we assume henceforth that  $H \le q <  N$.
By dyadic division, it suffices to show that if $J \le J' \le 2J$, $J' \le H$ and $ H \le q <  N$,
then
\begin{multline}\label{claim}
\sum_{J \le j \le J'} \bigg|\sum_{\substack{n=1\\ n \equiv u \bmod v}}^N \Lambda(n) e(\alpha jn) \bigg| \ll  \mathcal{L}_1^{7} \left( \frac{JN}{q^{1/2}} + JN^{3/4} + J^{1/2} N^{5/6}+ (JNq)^{1/2} + (JN)^{\varepsilon}q \right).
\end{multline}

We now recall Vaughan's identity \cite{GH} for the $\Lambda$ function.
It states that for real numbers $U$, $V$ both greater than $0$,
$$
\Lambda(n) = a_1(n) + a_2(n) + a_3(n) + a_4(n)
$$
where
\begin{eqnarray*}
a_1(n) = \begin{cases}
\Lambda(n) \text{ if } n \le U\\
0 \text{ otherwise, }
\end{cases}
&&
a_2(n) = - \sum_{\substack{mdr=n \\ m \le U, d\le V}} \Lambda(m) \mu(d),\\
a_3(n) = \sum_{\substack{hd =n, d\le V}} \mu(d) \log h & \text{ and } & a_4(n) = - \sum_{\substack{mk=n \\ m>U, k>V}} \Lambda(m) \sum_{\substack{d \mid k \\ d \le V}} \mu(d).
\end{eqnarray*}
We set $U = V= N^{1/3}$.
Substituting the above in the left hand side of \eqref{claim}, we get
$$
\sum_{J \le j \le J'} \bigg|\sum_{\substack{n=1\\ n \equiv u \bmod v}}^N \Lambda(n) e(\alpha jn) \bigg| \le S_1 + S_2 + S_3 + S_4
$$
where
\begin{multline*}
S_1 = \sum_{J \le j \le J'} \bigg|\sum_{\substack{n=1\\ n \equiv u \bmod v \\ n \le U}}^N \Lambda(n) e(\alpha jn) \bigg|, 
~~
S_2 =  \sum_{J \le j \le J'} \bigg|\sum_{\substack{n=1\\ n \equiv u \bmod v}}^N\phantom{m}  \sum_{\substack{mdr=n \\ m \le U, d\le V}} \Lambda(m) \mu(d)e(\alpha jn) \bigg| \\
S_3 =  \sum_{J \le j \le J'} \bigg|\sum_{\substack{n=1\\ n \equiv u \bmod v}}^N\sum_{\substack{hd =n, d\le V}} \mu(d) \log h~ e(\alpha jn) \bigg|~~ \text{ and }  \\
S_4 =  \sum_{J \le j \le J'} \bigg|\sum_{\substack{n=1\\ n \equiv u \bmod v}}^N \phantom{m} \sum_{\substack{mk=n \\ m>U, k>V}} \Lambda(m) \sum_{\substack{d \mid k \\ d \le V}} \mu(d)
e(\alpha jn) \bigg|.
\end{multline*}
We estimate $S_1$ trivially, to get
$
S_1 \ll JU.
$
Since $U = N^{1/3}$, we have the required bound.
To estimate $S_2$, we first recall that
\begin{multline*}
S_2  
=  \sum_{J \le j \le J'} \bigg| \sum_{\substack{m \le U, \\ d\le V \\(md,v)=1}}\sum_{\substack{r=1\\ r \equiv \overline{md}u \bmod v }}^{N/(md)}  \Lambda(m) \mu(d)e(\alpha jmdr) \bigg| 
 \le  \sum_{J \le j \le J'}  \sum_{\substack{m \le U \\ d\le V\\(md,v)=1}}  \Lambda(m) \bigg|\sum_{\substack{r=1\\ r \equiv \overline{md}u \bmod v }}^{N/(md)}  e(\alpha jmdr) \bigg|. \\ 
\end{multline*}
Let $0 < r_0 \le v-1$ be an integer which satisfies $ r_0 \equiv \overline{md}u \bmod v$.
We now set $r =  r_0 + kv$, to get
\begin{eqnarray*}
S_2 
& \le & \sum_{J \le j \le J'}  \sum_{\substack{m \le U \\ d\le V}}  \Lambda(m) \bigg|\sum_{\substack{k=0 }}^{N/(mdv) - r_0/v}  e(\alpha jmdkv) \bigg| 
\le 
 \sum_{J \le j \le J'}  \sum_{\substack{m \le U \\ d\le V}}  \Lambda(m) \max_{y \le N/(mdv)}  \bigg|\sum_{\substack{k=0 }}^{y}  e(\alpha jmd kv) \bigg|.
\end{eqnarray*}
Setting $s = jmd$, we get for any $\varepsilon >0$
$$
S_2 \ll (JUV)^{\varepsilon/2}  \sum_{ J \le s \le J'UV}  \max_{y \le NJ'/(sv)}  \bigg|\sum_{\substack{k=0 }}^{y}  e(\alpha skv) \bigg|.
$$
By Vaughan's type I estimate (\lemref{Vaughan1}), we get
$$
S_2 \ll (JUV)^{\varepsilon}  \left( \frac{NJ}{vq} + JUV + q \right).
$$
Since $UV = N^{2/3}$, for $\varepsilon$ sufficiently small, we have the required bound.
To estimate $S_3$, we note that
\begin{eqnarray*}
S_3 & = &  \sum_{J \le j \le J'} \bigg|\sum_{\substack{n=1\\ n \equiv u \bmod v}}^N\sum_{\substack{hd =n, d\le V}} \mu(d) \log h~ e(\alpha jn) \bigg| 
 =   \sum_{J \le j \le J'} \bigg|\sum_{\substack{hd \le N , d\le V \\ hd \equiv u\bmod v}} \mu(d) \log h~ e(\alpha jhd) \bigg|\\
& = & \sum_{J \le j \le J'} \bigg|\sum_{\substack{d \le V\\ (d,v)=1}} \mu(d) \sum_{\substack{h\le N/d \\ h \equiv \overline{d}u \bmod v}}  \log h~ e(\alpha jhd) \bigg|
 \le  \sum_{J \le j \le J'}  \sum_{\substack{d \le V\\ (d,v)=1}} \bigg| \sum_{\substack{h\le N/d \\ h \equiv \overline{d}u \bmod v}}  \log h~ e(\alpha jhd) \bigg|. 
\end{eqnarray*}
We now follow Vaughan's argument in \cite{RCV}. Note that $\log h = \int_1^h dt /t$.
$$
S_3 \le  \sum_{J \le j \le J'}  \sum_{\substack{d \le V\\ (d,v)=1}} \bigg| \sum_{\substack{h\le N/d \\ h \equiv \overline{d}u \bmod v}}  \int_1^h \frac{dt}{t}~ e(\alpha jhd) \bigg|
=
\sum_{J \le j \le J'}  \sum_{\substack{d \le V\\ (d,v)=1}} \bigg| \int_1^{N/d} \frac{dt}{t} \sum_{  \substack{t \le h\le N/d \\ h \equiv \overline{d}u \bmod v}}  e(\alpha jhd) \bigg|.
$$
Further
$$
S_3 \le \sum_{J \le j \le J'}  \sum_{\substack{d \le V\\ (d,v)=1}} \int_1^{N/d} \frac{dt}{t} \bigg|  \sum_{  \substack{t \le h\le N/d \\ h \equiv \overline{d}u \bmod v}}  e(\alpha jhd) \bigg|
\le 
 \int_1^{N} \frac{dt}{t}  \sum_{J \le j \le J'}  \sum_{\substack{d \le V\\ (d,v)=1}}\bigg|  \sum_{  \substack{t \le h\le N/d \\ h \equiv \overline{d}u \bmod v}}  e(\alpha jhd) \bigg|
$$
Let us choose an integer $0 < d_0 \le v-1$ such that $d_0  \equiv \overline{d}u \bmod v$.
Now setting $ h = d_0 + kv$, we get
$$
S_3 \le  \int_1^{N} \frac{dt}{t}  \sum_{J \le j \le J'}  \sum_{d \le V}\bigg|  \sum_{\frac{t- d_0}{v} \le k \le \frac{N/d-d_0}{v} }  e(\alpha jdkv) \bigg|
\le 
 \int_1^{N} \frac{dt}{t}  \sum_{J \le j \le J'}  \sum_{d \le V} ~\max_{y \le \frac{N}{dv}}~\bigg|  \sum_{\frac{t-d_0}{v} \le k \le y}  e(\alpha jdkv) \bigg|.
$$
We now combine the variables $j$ and $d$ by setting $jd=w$, to get
$$
S_3 \ll (JN)^{\varepsilon/2}  \int_1^{N} \frac{dt}{t}  \sum_{J \le w \le J'V} ~\max_{y \le \frac{NJ'}{wv}}~ \bigg|  \sum_{\frac{t-d_0}{v} \le k \le y }  e(\alpha wkv) \bigg|
$$
for any real number $\varepsilon >0$.
Using Vaughan's type I sum estimate (\lemref{Vaughan1}), we get
$$
S_3 \ll  (JN)^{\varepsilon} \left(\frac{JN}{v q} + JV +q \right).
$$
Since $V = N^{1/3}$, for $\varepsilon$ sufficiently small, we have the required bound.
We will now estimate the term $S_4$.
Recall that
$$
S_4 = \sum_{J \le j \le J'} \bigg| \sum_{\substack{n=1\\ n \equiv u \bmod v}}^N~ \sum_{\substack{mk = n \\ m>U, k>V}} \Lambda(m) \sum_{\substack{d \mid k \\ d \le V}} \mu(d) e(\alpha j n) \bigg|.
$$
We break the interval $[V,N/U]$ into smaller intervals of the the form $[2^s V, 2^{s+1}V] \cap  [V,N/U]$. 
Here $s$ is an integer greater than or equal to $0$. 
We define for real numbers $K,K' >0$ the sum
$$
S_{4}(K,K') := \sum_{J \le j \le J'} \bigg|  \sum_{\substack{ K < k \le K'\\ U < m \le N/k\\ mk \equiv u \bmod v}} \Lambda(m) \sum_{\substack{d \mid k \\ d \le V}} \mu(d) e(\alpha j mk) \bigg|.
$$
We observe that
\begin{eqnarray*}
|S_4| 
& \le & \sum_{0 < s < \log_2 (N/(UV))} \sum_{J \le j \le J'} \bigg|  \sum_{\substack{ 2^{s}V<k\le \min(2^{s+1}V, N/U) \\ U < m \le N/k\\ mk \equiv u \bmod v}} \Lambda(m) \sum_{\substack{d \mid k \\ d \le V}} \mu(d) e(\alpha j mk) \bigg|\\
& \le & \sum_{\substack{0 < s < \log_2 (N/(UV))}} |S_{4}(2^sV,\min(2^{s+1}V, N/U))|.
\end{eqnarray*}
There are $\rO(\log N)$ terms in the above sum.
We note that it suffices to show that 
$$
S_{4}(K,K') \ll
\mathcal{L}_1^{6} \left(\frac{JN}{q^{1/2}} + JN^{3/4} + J^{1/2} N^{5/6}+(JNq)^{1/2} \right)
$$
for $K=2^s V, K'= \min(2^{s+1}V, N/U)$ for all $0 < s < \log_2 N^{1/3}$. 
We will now estimate $S_{4}(K,K')$ for $K=2^s V, K'= \min(2^{s+1}V, N/U)$ for all $0 < s < \log_2 N^{1/3}$.
This will be done based on whether 
$V < K \ll N^{1/2} ~~\text{ or }~~ N^{1/2} \ll K < N/U.$
We set $a_m = \Lambda(m)$ and $b_k = \sum_{\substack{d \mid k \\ d \le V}} \mu(d)$.
We can now write
\begin{eqnarray*}
S_{4}(K,K') & = & \sum_{x_1 \bmod v}\sum_{\substack{J \le j \le J' \\ j \equiv x_1 \bmod v}} \bigg| \sum_{\substack{x_2 \bmod v \\ (x_2,v)=1}} ~ \sum_{\substack{ K < k \le K'\\ U < m \le N/k \\m \equiv x_2 \bmod v \\ k \equiv u\overline{x_2} \bmod v}}  a_m b_k e(\alpha j mk) \bigg|  \\
& \le & \sum_{x_1 \bmod v} \sum_{\substack{x_2 \bmod v \\ (x_2,v)=1}} \sum_{\substack{J \le j \le J' \\ j \equiv x_1 \bmod v}} \bigg|  \sum_{\substack{K < k \le K'\\ U < m \le N/k \\m \equiv x_2 \bmod v \\ k \equiv u\overline{x_2} \bmod v}} a_m b_k e(\alpha j mk) \bigg|  \\
& \le& v^2 \max_{(x_1,x_2,x_3) \bmod v} \sum_{\substack{J \le j \le J' \\ j \equiv x_1 \bmod v}} \bigg|  \sum_{\substack{ K < k \le K' \\ U< m \le N/k  \\ m \equiv x_2 \bmod v \\k \equiv  x_3 \bmod v}} a_m b_k e(\alpha j mk) \bigg|.
\end{eqnarray*}
For some fixed choice of classes $(x_1,x_2,x_3) \bmod v$ with $0 \le x_i \le v-1$ for $i \in \{1,2,3\}$ set
\begin{eqnarray*}
S_{4}(K, K', x_1,x_2,x_3) &:= &\sum_{\substack{J \le j \le J' \\ j \equiv x_1 \bmod v}} \bigg|  \sum_{\substack{ K< k \le K' \\ U < m \le N/k \\ m \equiv x_2 \bmod v\\ k \equiv  x_3 \bmod v}} a_m b_k e(\alpha j mk) \bigg|. \\
\end{eqnarray*}
Therefore, we have
$$
S_{4}(K, K', x_1,x_2,x_3) 
 = 
\sum_{\substack{J \le j \le J' \\ j \equiv x_1 \bmod v}} \bigg|  \sum_{\substack{ K < k \le K'\\ k \equiv x_3 \bmod v }} b_k \sum_{\substack{U < m \le N/k \\ m \equiv  x_2 \bmod v}} a_m e(\alpha j mk) \bigg|. \\
$$
We may write
$$
S_{4}( K, K', x_1,x_2,x_3):= \sum_{\substack{J \le j \le J' \\ j \equiv x_1 \bmod v}} c_j   \sum_{\substack{ K < k \le K'\\ k \equiv x_3 \bmod v }} b_k \sum_{\substack{U < m \le N/k \\ m \equiv  x_2 \bmod v}} a_m e(\alpha j mk)
$$
for some complex constants $c_j$, $J \le j \le J'$ with $|c_j|=1$.
Suppose that $V < K \ll N^{1/2}$.
We now set $w = jk$, to get
$$
S_{4}( K, K',  x_1,x_2,x_3):= \sum_{\substack{JK \le w \le J'K' \\ w \equiv x_1x_3 \bmod v}}   \sum_{\substack{ k \mid w\\ K < k \le K' \\ J \le w/k \le J'\\ k \equiv x_3 \bmod v \\ w/k \equiv x_1 \bmod v }} \left( c_{w/k}b_k \sum_{\substack{ U < m \le N/k \\m \equiv  x_2 \bmod v}} a_m e(\alpha wm) \right).
$$
We shall use $\tau_d(n)$ to denote the number of ordered ways of writing $n$ as a product of $d$ positive integers.
Now
\begin{eqnarray*}
|S_{4}( K, K',  x_1,x_2,x_3)| &\le & \sum_{\substack{JK \le w \le J'K' \\ w \equiv x_1x_3 \bmod v}}   \sum_{\substack{ k \mid w\\ K < k \le K' \\ J \le w/k \le J'\\ k \equiv x_3 \bmod v \\ w/k \equiv x_1 \bmod v }} \left( \tau_2(k) \left| \sum_{\substack{ U < m \le N/k \\m \equiv  x_2 \bmod v}} a_m e(\alpha wm) \right| \right) \\
& \le & 
  \sum_{\substack{JK \le w \le J'K' \\ w \equiv x_1x_3 \bmod v}}  \tau_3(w) \max_{y \le N/K} \left| \sum_{\substack{ U < m \le y \\m \equiv  x_2 \bmod v}} a_m e(\alpha wm) \right|  \\
& \le &
\sum_{\substack{JK \le w \le J'K' \\ w \equiv x_1x_3 \bmod v}}  \tau_3(w) \max_{y \le \frac{N/K - x_2}{v}} \left| \sum_{\substack{ \frac{U-x_2}{v} < m_1 \le y }} a_{x_2 +m_1v} e(\alpha wm_1v) \right|.  \\
\end{eqnarray*}
We have $$\sum_{m_1 \le \frac{N/K-x_2}{v}} | a_{x_2 +m_1v}|^2 \le \sum_{m \le N/K } \Lambda(m)^2 \ll \frac{N \log N}{K}.$$
    Using the fact that $\sum_{k \le X} \tau_3(k)^2 \ll X\log^8 X$, we may apply Vaughan's type II estimate (\lemref{Vaughan2}). 
    If $V < K \ll N^{1/2}$ we get
    \begin{eqnarray*}
  |S_{4}( K, K',  x_1,x_2,x_3)|        \ll
     \mathcal{L}_1^{6} (JK)^{1/2} (N/K)^{1/2} \left(\frac{JN}{q} + JK + \frac{N}{K} + q\right)^{1/2}\\
     \ll
       \mathcal{L}_1^{6}  \left(\frac{JN}{q^{1/2}} + J(KN)^{1/2} + \frac{J^{1/2}N}{K^{1/2}} + (JNq)^{1/2}\right)\\
       \ll
       \mathcal{L}_1^{6}  \left(\frac{JN}{q^{1/2}} + JN^{3/4} + J^{1/2}N^{5/6} + (JNq)^{1/2}\right).
     \end{eqnarray*}

Now we consider the case where $ N^{1/2} \ll K < N/U$.
Recall
\begin{multline*}
S_{4}( K, K', x_1,x_2,x_3):=\sum_{\substack{J \le j \le J' \\ j \equiv x_1 \bmod v}} c_j   \sum_{\substack{ K < k \le K'\\ k \equiv x_3 \bmod v }} b_k \sum_{\substack{U < m \le N/k \\ m \equiv  x_2 \bmod v}} a_m e(\alpha j mk)\\
= \sum_{\substack{ K < k \le K'\\ k \equiv x_3 \bmod v }} b_k \sum_{\substack{J \le j \le J' \\ j \equiv x_1 \bmod v}} c_j   \sum_{\substack{ U < m \le N/k \\m \equiv  x_2 \bmod v}} a_m e(\alpha j mk).
\end{multline*}
We set $w = jm$ to get
$$
S_{4}( K,K', x_1,x_2,x_3) = \sum_{\substack{ K < k \le K'\\ k \equiv x_3 \bmod v }} b_k \sum_{\substack{JU \le w \le J'N/k \\ w \equiv x_1x_2 \bmod v}}   \sum_{\substack{m \mid w \\ U < m \le N/k \\ J \le w/m \le J' \\ m \equiv x_2 \bmod v\\ w/m \equiv x_1\bmod v}} a_m c_{w/m} e(\alpha wk).
$$
For a real number $t$, let us set 
$$
T_{1,t} = \sum_{\substack{ K < k \le K'\\ k \equiv x_3 \bmod v }} b_k k^{it} \sum_{\substack{JU \le w \le J'N/k \\ w \equiv x_1x_2 \bmod v}}   \sum_{\substack{m \mid w \\ U < m \le N/K \\ J \le w/m \le J' \\ m \equiv x_2 \bmod v\\ w/m \equiv x_1\bmod v}} a_m m^{it} c_{w/m} e(\alpha wk)
$$
and 
$$
T_{2} = \sum_{\substack{ K < k \le K'\\ k \equiv x_3 \bmod v }} b_k  \sum_{\substack{JU \le w \le J'N/k \\ w \equiv x_1x_2 \bmod v}}   \sum_{\substack{m \mid w \\ U < m \le N/K \\ J \le w/m \le J' \\ m \equiv x_2 \bmod v\\ w/m \equiv x_1\bmod v}} a_m c_{w/m} e(\alpha wk).
$$

Setting $\gamma = \log (mk)$ and $\rho = \log (N+\frac12)$ in \lemref{perron}, we get that 
\begin{multline*}
\pi S_{4}( K, K', x_1,x_2,x_3) =\\
 \sum_{\substack{ K < k \le K'\\ k \equiv x_3 \bmod v }} b_k \sum_{\substack{JU \le w \le J'N/k \\ w \equiv x_1x_2 \bmod v}}   \sum_{\substack{m \mid w\\  U < m \le N/K \\ J \le w/m \le J'\\ m \equiv x_2 \bmod v\\ w/m \equiv x_1\bmod v}} a_m c_{w/m} e(\alpha wk) \bigg( \int_{-L}^L e^{i t\log(mk) } \frac{ \sin( \log \left(N+\frac12 \right) t)}{t} dt  \\
 ~+ ~  \rO\left(L^{-1} |\rho -\gamma|^{-1}  \right) \bigg).
\end{multline*}
Note that if $mk \le N$ we have $\left|\log\frac{(N+\frac12)}{mk} \right| \ge \log(1 + \frac{1}{2N}) \gg \frac{1}{N}$.
Further, if $mk \ge N+1$, we get $\left|\log\frac{(N+\frac12)}{mk} \right| = \log\frac{mk}{(N+\frac12)} \ge \log(1 + \frac{1}{2N+1}) \gg \frac{1}{N}$.
In either case, we have  $\left|\log\frac{(N+\frac12)}{mk} \right| \gg \frac{1}{N}$. 
So the total error term above is
$
\ll \frac{N}{L} \left| T_2 \right|.
$
We note that
$$
 \int_{-L}^L e^{i t\log(mk) } \frac{ \sin (\log(N+\frac12) t)}{t} dt 
 =
\left( \int_{[-L,-\frac{1}{L}]\cup [\frac{1}{L}, L]} + \int_{[-\frac{1}{L},  \frac{1}{L}]} \right) e^{i t \log(mk) } \frac{ \sin (\log(N+\frac12) t)}{t} dt.
$$
For
$$
I_1=\sum_{\substack{ K < k \le K'\\ k \equiv x_3 \bmod v }} b_k \sum_{\substack{JU \le w \le J'N/k \\ w \equiv x_1x_2 \bmod v}}   \sum_{\substack{m \mid w \\ U < m \le N/K \\ J \le w/m \le J' \\ m \equiv x_2 \bmod v\\ w/m \equiv x_1\bmod v}} a_m c_{w/m} e(\alpha wk) \int_{[-L,-\frac{1}{L}] \cup [\frac{1}{L}, L]} e^{i t\log(mk) } \frac{ \sin (\log(N+\frac12) t)}{t} dt 
$$
we use the bound $|\sin x| \le 1$ to get
\begin{eqnarray*}
|I_1| \le  \max_{|t| \le L} \left|T_{1,t}\right|  \int_{[-L,-\frac{1}{L}] \cup [\frac{1}{L}, L]} \frac{1}{|t|} dt 
~~\ll~~ \log L \max_{|t| \le L} \left|T_{1,t}\right|.
\end{eqnarray*}
For
$$
I_2=\sum_{\substack{ K < k \le K'\\ k \equiv x_3 \bmod v }} b_k \sum_{\substack{JU \le w \le J'N/k \\ w \equiv x_1x_2 \bmod v}}   \sum_{\substack{m \mid w\\ U < m \le N/K \\ J \le w/m \le J' \\ m \equiv x_2 \bmod v\\ w/m \equiv x_1\bmod v}} a_m c_{w/m} e(\alpha wk) \int_{[-\frac{1}{L}, \frac{1}{L}]} e^{it\log(mk) } \frac{ \sin (\log(N+\frac12) t)}{t} dt 
$$
we use the bound $|\sin x| \le |x|$ to get
\begin{eqnarray*}
|I_2| \le  \max_{|t| \le L} \left|T_{1,t}\right|  \int_{[-\frac{1}{L}, \frac{1}{L}]} \log\left(N+\frac12 \right) dt 
\ll \frac{ \log \left(N+\frac12\right)}{L} \max_{|t| \le L} \left|T_{1,t}\right|.
\end{eqnarray*}
Setting $L = N^2$, we get
$$
S_{4}( M, M', x_1,x_2,x_3) \ll \log N \max_{|t| \le N^2} \left|T_{1,t}\right| 
+ 
\frac{1}{N}  \left| T_2 \right|.
$$
We now estimate $|T_{1,t}|$ and $|T_2|$.
We have
$$
|T_{1,t}| \le \sum_{\substack{ K < k \le K'\\ k \equiv x_3 \bmod v }} |b_k | \max_{y\le J'N/K} \left|\sum_{\substack{JU \le w \le y \\ w \equiv x_1x_2 \bmod v}}  d_w e(\alpha wk) \right|
\text{ where }
d_{w} =  \sum_{\substack{m \mid w \\ U < m \le N/K \\ J \le w/m \le J' \\ m \equiv x_2 \bmod v\\ w/m \equiv x_1\bmod v}} a_m m^{it} c_{w/m}.$$
We have
\begin{eqnarray*}
 \max_{y\le J'N/K} \left|\sum_{\substack{JU \le w \le y \\ w \equiv x_1x_2 \bmod v}}  d_w e(\alpha wk) \right|
 \le  
 \max_{y \le \frac{J'N/K - x_1x_2}{v}} \left| \sum_{\substack{\frac{JU-x_1x_2}{v} \le w_1 \le y}}   d_{x_1x_2 +w_1v} e(\alpha w_1kv) \right|.
\end{eqnarray*}
We note that for any real $X>0$, $\sum_{k \le X} \tau(k)^2 \ll X \log^3 X$ and
$$\sum_{w_1 \le \frac{J'N/K - x_1x_2}{v}} | d_{x_1x_2 +w_1v}|^2 \le \sum_{w \le J'N/K } (\sum_{k \mid w} \Lambda(k))^2 \ll \frac{JN\log^2 N}{K}.$$
We may apply Vaughan's type II estimate (\lemref{Vaughan2}). to get
   \begin{multline*}
|T_{1,t}| \ll  \mathcal{L}_1^{4} (JN/K)^{\frac12} (K)^{\frac12} \left(\frac{JN}{q} + K + \frac{JN}{K} + q\right)^{\frac12} 
     \ll 
     \mathcal{L}_1^{4} \left(\frac{JN}{q^{1/2}} + (JNK)^{1/2} + \frac{JN}{K^{1/2}} + (JNq)^{1/2}\right).
     \end{multline*}
Further
$$
|T_2| \le \sum_{\substack{ K < k \le K'\\ k \equiv x_3 \bmod v }} |b_k | \max_{y\le J'N/K} \left|\sum_{\substack{JU \le w \le y \\ w \equiv x_1x_2 \bmod v}}   d_w e(\alpha wk) \right|
\text{ where }
d_{w} =  \sum_{\substack{m \mid w \\ U < m \le N/K \\ J \le w/m \le J' \\ m \equiv x_2 \bmod v\\ w/m \equiv x_1\bmod v}} a_m  c_{w/m}.$$
We have
\begin{eqnarray*}
 \max_{y\le J'N/K} \left|\sum_{\substack{JU \le w \le y \\ w \equiv x_1x_2 \bmod v}}  d_w e(\alpha wk) \right|
 \le 
 \max_{y \le \frac{J'N/K - x_1x_2}{v}} \left| \sum_{\substack{\frac{JU-x_1x_2}{v} \le w_1 \le y}}   d_{x_1x_2 +w_1v} e(\alpha w_1kv) \right|.
\end{eqnarray*}
We note that $\sum_{k \le X} \tau(k)^2 \ll X\log^3 X$ and
$$\sum_{w_1 \le \frac{J'N/K - x_1x_2}{v}} | d_{x_1x_2 +w_1v}|^2 \le \sum_{w \le J'N/K } \log(w)^2 \ll JN(\log N)^2/K.$$
We may apply Vaughan's type II estimate (\lemref{Vaughan2}). to get
   \begin{multline*}
|T_2| \ll  \mathcal{L}_1^{4} (JN/K)^{\frac12} (K)^{\frac12} \left(\frac{JN}{q} + K + \frac{JN}{K} + q\right)^{\frac12} 
     \ll 
     \mathcal{L}_1^{4} \left(\frac{JN}{q^{1/2}} + (JNK)^{1/2} + \frac{JN}{K^{1/2}} + (JNq)^{1/2}\right).
     \end{multline*}
     Since $N^{1/2} \ll K < N/U$, we have 
     \begin{equation*}
\max_{|t| \le N^2} |T_{1,t}|, |T_2|      \ll  \mathcal{L}_1^{4} \left(\frac{JN}{q^{1/2}} + J^{1/2} N^{5/6} + JN^{3/4} + (JNq)^{1/2}\right).
     \end{equation*}
     Therefore if $N^{1/2} \ll K < N/U$,
      $$
  |S_{4}( K, K',  x_1,x_2,x_3)| \ll     \mathcal{L}_1^{5} \left(\frac{JN}{q^{1/2}} + J^{1/2}N^{5/6} + JN^{3/4} + (JNq)^{1/2}\right).
     $$
     This completes the proof.
       \end{proof}

\begin{proof}[Proof of \thmref{asymp}]
We know that
$$
f_{\Delta}(\theta) - 2\Delta \le \sum_{1 \le |h| \le H} c_h^+ e(h\theta) + \frac{1}{H+1}
\text{ with }
|c_h^+| \ll \min \left( \Delta + \frac{1}{H}, \frac{1}{|h|} \right).
$$
Therefore we have
\begin{multline*}
\sum_{\substack{n=1 \\ n \equiv u \bmod v}}^N \Lambda(n) (f_{\Delta}(\alpha n -\beta) - 2\Delta) \le \sum_{\substack{n=1 \\ n \equiv u \bmod v}}^N \Lambda(n) \sum_{1 \le |h| \le H} c_h^{+} e(h(\alpha n -\beta)) + \frac{1}{H+1}\sum_{\substack{n=1 \\ n \equiv u \bmod v}}^N \Lambda(n).
\end{multline*}
Consider the first term
$$
 \sum_{\substack{n=1 \\ n \equiv u \bmod v}}^N \Lambda(n) \sum_{1 \le |h| \le H} c_h^+ e(h(\alpha n -\beta))  \ll 
 \sum_{|h|=1}^H \min \left( \Delta + \frac{1}{H}, \frac{1}{|h|} \right) \bigg|\sum_{\substack{n=1 \\ n \equiv u \bmod v}}^N \Lambda(n) e(h\alpha n) \bigg|.
$$
We note that for $H = 1 +\floor{Nq/\Delta}$, we have
$$
 \sum_{\substack{n=1 \\ n \equiv u \bmod v}}^N \Lambda(n) \sum_{1 \le |h| \le H} c_h^+ e(h(\alpha n -\beta))  \ll 
 \sum_{|h|=1}^H \min \left( \Delta, \frac{1}{|h|} \right) \bigg|\sum_{\substack{n=1 \\ n \equiv u \bmod v}}^N \Lambda(n) e(h\alpha n) \bigg|.
$$
For the error term, we have
\begin{eqnarray*}
\frac{1}{H+1} \sum_{\substack{n=1 \\ n \equiv u \bmod v}}^N \Lambda(n)   \ll \frac{N}{H} \ll \frac{\Delta}{q}.
\end{eqnarray*}

Now we move on to the exponential sum.
Let
$$
S(m) = \sum_{1 \le h \le m}\bigg| \sum_{\substack{ n=1\\ n \equiv u \bmod v}}^N \Lambda(n) e(\alpha hn)\bigg|. 
$$
We have, by partial summation
\begin{multline*}
 \sum_{h=1}^H \min \left( \Delta, \frac{1}{h} \right) \bigg|\sum_{\substack{n=1 \\ n \equiv u \bmod v}}^N \Lambda(n) e(h\alpha n) \bigg|
  = 
 \frac{S(H)}{H} + \int_{\Delta^{-1}}^H \frac{S(t)}{t^2} dt \\
  =  \mathcal{L}_1^{7} (N/q^{1/2} + N^{3/4} + N^{5/6}/H^{1/2} + (Nq)^{1/2}/H^{1/2} + N^{\varepsilon}q/H^{1-\varepsilon}  \\
 +   \mathcal{L}_1(N/q^{1/2} + N^{3/4}) + N^{5/6}\Delta^{1/2} +  (Nq)^{1/2} \Delta^{1/2}+ N^{\varepsilon}q \Delta^{1-\varepsilon} )\\
  \ll   \mathcal{L}_1^{8}(N/q^{1/2} + N^{3/4} + N^{5/6}\Delta^{1/2} + (\Delta Nq)^{1/2} + N^{\varepsilon}q \Delta^{1-\varepsilon}).
 \end{multline*}
 We note that for the sum
 $$
 \sum_{h=-1}^{-H} \min \left( \Delta, \frac{1}{|h|} \right) \bigg|\sum_{\substack{n=1 \\ n \equiv u \bmod v}}^N \Lambda(n) e(h\alpha n) \bigg|
 =  \sum_{h=1}^{H} \min \left( \Delta, \frac{1}{h} \right) \bigg|\sum_{\substack{n=1 \\ n \equiv u \bmod v}}^N \Lambda(n) e(-h\alpha n) \bigg|,
 $$
we may consider 
$$
S''(m) = \sum_{1 \le h \le m}\bigg| \sum_{\substack{ n=1\\ n \equiv u \bmod v}}^N \Lambda(n) e(-\alpha hn)\bigg|.
$$
Since $|-\alpha v - (-a)/q| < 1/q^2$, the same calculation gives us the required bound.
\end{proof}

\begin{proof} 
[Proof of Corollary \ref{mainthm}]
Set  $\epsilon$ be a real number in $(0,2/5)$, $N = q^2$ and $\Delta = CN^{-1/4} \log^{8} N$ for a sufficiently large constant $C>0$  in \thmref{asymp}.
This gives the first part of the Corollary. 
For the second part of the corollary, 
we have by \thmref{asymp} that
$$
\sum_{\substack{n=1 \\ n \equiv u \bmod v}}^N \Lambda(n) (f_{\Delta(N)}(\alpha n - \beta) - 2\Delta) \ll_v  \mathcal{L}^8  ( N/q^{1/2} + N^{3/4} + N^{5/6}\Delta^{1/2} + (\Delta Nq)^{1/2} + N^{\epsilon}q \Delta^{1-\epsilon}).
$$
We will use $\pi(N,u,v)$ to denote the number of primes $\ell \le N$ satisfying $\ell \equiv u \bmod v$.
For the last part of the corollary, for $v$ fixed and $N \to \infty$, we note by partial summation that
\begin{eqnarray*}
\sum_{\substack{p \le N \text{ prime } \\ p \equiv u \bmod v \\ \|\alpha p -\beta\| < \Delta}}^N 1 - 2\Delta \pi(N,u,v)
& = & 
\sum_{\substack{p \le N  \text{ prime } \\ p \equiv u \bmod v \\ \|\alpha  p -\beta\| < \Delta}}
\frac{\Lambda(p)}{\log p}  -  2\Delta\sum_{\substack{p \le N  \text{ prime } \\ p \equiv u \bmod v}} \frac{\Lambda(p)}{\log p} \\
& =&
\sum_{\substack{n \le N  \\ n \equiv u \bmod v \\ \|\alpha  n -\beta\| < \Delta}} \frac{\Lambda(n)}{\log n}  -  2\Delta\sum_{\substack{n \le N  \\ n \equiv u \bmod v}} \frac{\Lambda(n)}{\log n} +{\rm O}(N^{1/2} \log N) \\
& = &\frac{\sum_{\substack{n=1 \\ n \equiv u \bmod v}}^N \Lambda(n) (f_{\Delta}(\alpha n - \beta)-2\Delta)}{\log N} +{\rm O}(N^{1/2} \log N) \\
&+ &\int_{2}^N \frac{ \sum_{\substack{n=1 \\ n \equiv u \bmod v}}^t \Lambda(n) (f_{\Delta}(\alpha n - \beta)-2\Delta)}{t \log^2 t} dt\\
& = &
 {\rm O}_v\left(\int_{2}^N \frac{ \mathcal{L}^8  (t/q^{1/2} + t^{3/4} + t^{5/6}\Delta^{1/2} + (\Delta tq)^{1/2} +  t^{\epsilon}q \Delta^{1-\epsilon})}{t \log^2 t} dt \right)\\
& + & {\rm O}_v(\mathcal{L}^7  ( N/q^{1/2} + N^{3/4} + N^{5/6}\Delta^{1/2} + (\Delta Nq)^{1/2} + N^{\epsilon}q \Delta^{1-\epsilon})).\\
\end{eqnarray*}
For each of the remaining integrals, we proceed in the same manner as seen below. As $N \to \infty$,
\begin{eqnarray*}
\int_{2}^N \frac{ \mathcal{L}^8   t^{3/4}}{t\log^2 t}dt   \ll N^{3/4} \mathcal{L}^6 \int_2^N \frac{dt}{t} \ll N^{3/4}\mathcal{L}^7 , 
&~&  \int_{2}^N \frac{\mathcal{L}^8 t^{5/6}\Delta^{1/2}}{t \log^2 t} dt \ll N^{5/6} \Delta^{1/2} \mathcal{L}^7 ,\\
\int_{2}^N \frac{\mathcal{L}^8 dt}{ \log^2 t} \ll N\mathcal{L}^6,~ \int_{2}^N \frac{\mathcal{L}^8 (\Delta tq)^{1/2}}{t\log^2 t}dt \ll (Nq \Delta)^{1/2} \mathcal{L}^7  & \text{ and } & \int_{2}^N \frac{\mathcal{L}^8t^{\epsilon}q \Delta^{1-\epsilon}}{t \log^2 t} dt \ll N^{ \epsilon} q\Delta^{1-\epsilon}\mathcal{L}^7.
\end{eqnarray*}
Therefore, for fixed $v$, as  $N \to \infty$,
\begin{multline*}
\{ 1 \le p \le N~ : ~p \text{ prime},~p \equiv u \bmod v,~ \| \alpha p - \beta\| < \Delta(N)\}   = 
2\Delta(N) \pi(N,u,v) \\
 +  
{\rm O}_v(\mathcal{L}^7  ( N/q^{1/2} + N^{3/4} + N^{5/6}\Delta^{1/2} + (\Delta Nq)^{1/2} + N^{\epsilon}q \Delta^{1-\epsilon})).
\end{multline*}

\end{proof}
\section{Examples of irrational numbers  all of whose convergents have denominators prime and congruent to $u \bmod v$}\label{examples}
In this section, we construct an uncountable set $S \subset \R \setminus \Q$
such that for every $\gamma \in S$, all  convergents of $\gamma$ have denominators which are primes and congruent to $u \bmod v$.
In order to construct $S$, we begin with a set $\mathcal{P}$ defined as follows:
\begin{eqnarray*}
\mathcal{P}:=
\left\{ (q_i)_{i=-1}^{\infty} ~:~ \begin{aligned} 
 q_{-1}=0, q_0 = 1, \text{ for all } n \ge 1, q_n \text{ is prime},\\ q_n \equiv u \bmod v ,  q_n > q_{n-1}, q_{n} \equiv q_{n-2} \bmod q_{n-1} 
 \end{aligned} \right\}.
\end{eqnarray*}
The Set $S$ is now given by 
\begin{eqnarray*}
S := \left\{ \gamma = [a_0;a_1,a_2,\ldots] ~:~ \begin{aligned}  a_0 \in \mathbb{N}, \text{ there exists  a }  (q_i)_{i=-1}^{\infty} \in \mathcal{P} \text{ such that} \\ \text{ for all } n \ge 1,  a_n = (q_n - q_{n-2})/q_{n-1}  \end{aligned} \right\}.
\end{eqnarray*}
From the theory of continued fractions (see \cite{YB}), it is known
that for $\gamma
= [a_0;a_1,a_2,\ldots] \in S$,
 $q_n$ is the denominator of the $n$-th
 convergent to $\gamma$.
By a result of Vahlen (see page 8 of \cite{YB}), infinitely many $q_n$ satisfy
$$
\left\| \gamma{q_n} \right\| < \frac{1}{2q_n}.
$$
For any $\gamma \in S$, since the $q_n$'s are prime by construction and satisfy the congruence $u \bmod v$, we have infinitely many primes $\ell$ satisfying
$$
\|\gamma \ell\| < \frac{1}{\ell} ~~~\text{ and } ~~~ \ell \equiv u\bmod v.
$$
\begin{lem}
The set  $S$ is uncountable.
\end{lem}
\begin{proof}
Consider the set $S_1 = \{\gamma = [a_0;a_1,a_2,\ldots] \in S ~:~ a_0=1\}$. 
The set $S_1$ is in bijection with the set $\mathcal{P}$.
We will show that $\mathcal{P}$ is uncountable.
To derive a contradiction, we suppose that $\mathcal{P}$ is countable and enumerate the elements of $\mathcal{P}$
by $\{c_0, c_1,\ldots\}$. 
We shall denote the sequence in $\mathcal{P}$ corresponding to $c_n$ by $(q_{n,i})_{i=-1}^{\infty}$.
We construct a new element of  $\mathcal{P}$ which is not equal to any of the $c_i$
for $i \in \mathbb{N}\bigcup \{0\}$.  
We set $\tilde{q}_{-1} =0$ and $\tilde{q}_0 =1$.
For $n \ge 1$, we inductively choose a prime $\tilde{q}_n$ such that 
\begin{itemize}
\item $\tilde{q}_n > \max(\tilde{q}_{n-1}, q_{n,n})$,
\item  $\tilde{q}_n \equiv u \bmod v$ \text{ and }
\item $\tilde{q}_{n} \equiv \tilde{q}_{n-2} \bmod \tilde{q}_{n-1}$.
\end{itemize}
For $n=1$, $\tilde{q}_{n-1} = 1$ and by Dirichlet's theorem on primes in arithmetic progressions,
there are infinitely many prime choices for $\tilde{q}_1 \equiv u \bmod v$. This means we can choose a prime $\tilde{q}_1$ satisfying the above conditions.
Further $(\tilde{q}_1, v )=1$.
For $n \ge 2$, if we have $(\tilde{q}_{n-1},v)=1$, by an application of Chinese remainder theorem
 and Dirichlet's theorem on primes in arithmetic progressions, we can choose a prime $\tilde{q}_n$
 satisfying the three conditions above.
This further implies that any such choice of $\tilde{q}_n$ satisfies $(\tilde{q}_n,v)=1$. 
Therefore the sequence $(\tilde{q}_i)_{i=-1}^{\infty} \in \mathcal{P}$ but not equal to any of the $c_i$  for $i \in \mathbb{N}\bigcup \{0\}$.  
By this version of Cantor's diagonal argument $S$  is uncountable, thereby proving the lemma.
\end{proof}
We now show that there exist infinitely many irrational numbers represented by elements of $S$
which are not Liouville. To do so, we define the irrationality measure of a real number and prove the following quantitative irrationality criterion.

\begin{defn}
Given a real number $\delta$ we define the irrationality measure of
$\delta$ to be
$$
\inf_{\lambda \in \mathbb{R}} \left\{ \lambda \ge 0 : \|\delta n \|< \frac{1}{n^{\lambda-1}} \text{ has finitely many solutions where } n \text{ is a positive integer } \right\}.
$$
We shall denote it by ${\rm Irr}(\delta)$.
\end{defn}
Recall that the irrationality measure of a Liouville number is infinity.
The following is a generalisation of lemma 1.35 of \cite{WZ}.
\begin{lem}\label{quantirr}
For a given $\delta \in \mathbb{R}$ and a positive real constant $M$, we suppose that there exists an infinite sequence of rationals $a_n/q_n$
such that
$$
|q_n\delta - a_n| =e^{-\psi(n)},~  n \in \mathbb{N}
$$
where $\psi$ is a  positive real valued arithmetic function which satisfies
\begin{enumerate}
\item $\lim_{n \to \infty} \psi(n)  = \infty$,
\item $ \limsup_{n \to \infty}  \frac{\psi(n+1)}{\psi(n)} \le M$, 
\item $\rho = \limsup_{n \to \infty} \frac{\log q_n}{\psi(n)} >0$. 
\end{enumerate}
Then ${\rm Irr}(\delta) \le 2(1+\rho)M$.
\end{lem}

\begin{proof}
By condition (2) and (3), we have for any $\epsilon > 0$, an $N_{\epsilon} \in \mathbb{N}$ such that for $n \ge N_{\epsilon}$
$$
 \frac{\psi(n)}{\psi(n-1)} \le (M+\epsilon) ~~\text{ and }~~ q_n  \le e^{\psi(n)(\rho +\epsilon)}.
$$
Let $a$ be an integer.
Since $e^{\psi(n)} \to \infty$,  given any  $q \in \mathbb{N}_{> 2}$ with $2q \ge e^{\psi(N_{\epsilon})}$ we choose the smallest
natural number $n \ge N_{\epsilon}+1$ and  such that  $ 2q \le e^{\psi(n)}$.
We have
$$
|(q\delta - a)q_n| = |q(q_n\delta - a_n) + (qa_n - aq_n)|.
$$
If $|qa_n-aq_n| \neq 0$, then it is at least $1$.
In this case
$$
|(q\delta - a)q_n|  \ge 1 - qe^{-\psi(n)}.
$$
If $|qa_n-aq_n| = 0$, then 
$$
|(q\delta - a)q_n| = qe^{-\psi(n)}.
$$
Since $2q \le e^{\psi(n)}$, we have
$$
\bigg| \delta - \frac{a}{q} \bigg| \ge \frac{1}{q_ne^{\psi(n)}}
\ge  \frac{1}{e^{(1+\rho + \epsilon)\psi(n)}}
 \ge \frac{1}{e^{ (1+\rho + \epsilon)(M +\epsilon)\psi(n-1)}}.
$$
Since $2q \ge e^{\psi(n-1)}$, 
if we set $(1+\rho)\epsilon + M\epsilon +\epsilon^2$ to be $\epsilon_1$, we get
$$
\bigg| \delta - \frac{a}{q} \bigg| \ge \frac{1}{(2q)^{((1+\rho)M + \epsilon_1)}} > \frac{1}{q^{2((1+\rho)M + \epsilon_1)}}.
$$
Since $a$ is an arbitrary integer, we have that 
$$
\|\delta q\| > \frac{1}{q^{2((1+\rho)M + \epsilon_1)-1}}.
$$
For any given $\epsilon$, the above inequality holds for $q \ge e^{\psi(N_{\epsilon})}/2$.
So the inequality 
$$
\|\delta q\| < \frac{1}{q^{2((1+\rho)M + \epsilon_1)-1}},
$$
has finitely many solutions $q \in \mathbb{N}$.
Therefore ${\rm Irr}(\delta) \le 2((1+\rho)M + \epsilon_1)$. Since $\epsilon > 0$ is arbitrary,
we have the result.
\end{proof}
We recall here a theorem due to T. Xylouris \cite{TX}.
\begin{thm}
Let $b$ and $s >1$ be two positive integers with $(b,s)=1$. There exists
a constant $Q_0>0$ such that the least
prime congruent to $b \bmod s $ is less than $s^{5.18}$ for $s \ge Q_0$.
\end{thm}
We now construct a subset $T$ of $S$.
 To do so, we begin with a subset $\mathcal{P}_1$ of $\mathcal{P}$ defined as follows:
\begin{eqnarray*}
\mathcal{P}_1:=
\left\{ \mathcal{A} = (q_i)_{i=-1}^{\infty} 
 \in \mathcal{P} ~:~  \begin{aligned} \text{ there exists } n_{\mathcal{A}} \ge 3 \text{ such that for all } n \ge n_{\mathcal{A}}, \\  ~q_{n} \text{ is the least prime satisfying  } q_n \equiv u \bmod v\\ \text{ and } q_n \equiv q_{n-2} + q_{n-1} \bmod q_{n-1}q_{n-2}  \end{aligned}
 \right\}.
\end{eqnarray*}
Given $\mathcal{A} = (q_i)_{i=-1}^{\infty} 
 \in \mathcal{P}$, for $n \ge 3$, $q_{n-1}$ and $q_{n-2}$ are at least $2$.
 We also have $q_{n-1} > q_{n-2}$. Then
 $$q_{n-1} <q_{n-2} + q_{n-1} < 2q_{n-1} \le q_{n-1}q_{n-2}.$$
 This implies that the least prime congruent to $u \bmod v$ and $q_{n-2} + q_{n-1} \bmod q_{n-1}q_{n-2}$ is strictly greater than $q_{n-1}$. Therefore $\mathcal{P}_1$ is non-empty.
The Set $T$ is now given by 
\begin{eqnarray*}
T := \left\{ \gamma = [a_0;a_1,a_2,\ldots] ~:~ \begin{aligned}  a_0 \in \mathbb{N}, \text{ there exists  a }  (q_i)_{i=-1}^{\infty} \in \mathcal{P}_1 \text{ such that} \\ \text{ for all } n \ge 1,  a_n = (q_n - q_{n-2})/q_{n-1}  \end{aligned} \right\}.
\end{eqnarray*}
Given $(q_i)_{i=-1}^{\infty} 
 \in \mathcal{P}$, since $q_n \to \infty$ as $n \to \infty$,  $q_n \ge Q_0$ for all $n$ sufficiently large. 
\begin{lem}
The set $T$ does not contain any Liouville numbers.
\end{lem}
\begin{proof}
Let $\gamma \in T$ and $p_n/q_n$ be the n-th convergent to $\gamma$.
For $n$ sufficiently large (say $n \ge N_{\gamma}$), 
we have that 
$$
q_{n} \text{ is the least prime satisfying  } q_n \equiv u \bmod v\\ \text{ and } q_n \equiv q_{n-2} + q_{n-1} \bmod q_{n-1}q_{n-2}
$$
and by the above result of Xylouris, we know that the least prime 
satisfying
 $$
 q_{n-1} +q_{n} \bmod q_{n-1}q_{n} \text{ and } u \bmod v
$$
is at most $(vq_{n-1}q_{n})^{5.18} \le (vq_{n})^{10.36} $. 
Therefore for $n > N_{\gamma}$, we have $q_{n+1} \le (vq_{n})^{10.36}$.
We know
that for any convergent $p_n/q_n$ of $\gamma \in T$, we have
$$
 \frac{1}{q_n+q_{n+1}} <  |q_n \alpha - p_n| < \frac{1}{q_{n+1}}.
$$
So if we set $\psi(n) = - \log  |q_n \alpha - p_n|$, we have
$$
\log q_{n+1} <  \psi(n) < \log( q_n + q_{n+1}).
$$
Therefore for $n >  N_{\gamma}$, we have
$$
\frac{\psi(n+1)}{\psi(n)} < \frac{ \log( q_{n+1} + q_{n+2})}{\log q_{n+1}} \le \frac{\log ( (vq_{n})^{10.36}+ (v(vq_n)^{10.36})^{10.36})}{\log  q_{n}} \le   \frac{\log (2v^{10.36} (vq_{n})^{107.33})}{\log  q_{n}} .
$$
So we set $M=107.33$. Finally we have 
$$
\frac{\log q_n}{\psi(n)}  > \frac{\log q_n}{\log (q_n + q_{n+1})} > \frac{\log q_n}{\log (q_n + (vq_{n})^{10.36})} >  \frac{\log q_n}{\log 2v^{10.36} +  \log q_{n}^{10.36}}.
$$
So we get that $\rho > 1/(10.36)$.
Further
$$
\frac{\log q_n}{\psi(n)}  < \frac{\log q_n}{\log (q_{n+1})} < 1.
$$

This implies that ${\rm Irr}(\alpha) \le 429.32$. Since the irrationality
measure of a Liouville number is $\infty$, these numbers
cannot be Liouville.
\end{proof}

\section*{Acknowledgements}
\noindent
The authors would like to thank Prof. Sanoli Gun for suggesting the problem. The authors  also thank Prof. Sanoli Gun and Prof. Purusottam Rath for their continued guidance throughout
this project. The first author would like to thank the NBHM postdoctoral fellowship program (fellowship number: 0204/10/(8)/2023/R\&D-II/2778) and the Chennai Mathematical Institute for 
financial support. Both the authors are  grateful to the Chennai Mathematical support for academic support. The second author would like to thank IISER Thiruvananthapuram and the
 Inspire Faculty fellowship program (fellowship number: DST/INSPIRE/04/2022/001064) for financial support. Both the authors would like to thank IISER Thiruvananthapuram for the
 wonderful working environment.

\end{document}